\documentclass[12pt]{amsart}
\usepackage{a4wide}
\usepackage{enumerate}
\usepackage{graphicx,tikz}
\usepackage{color}
\usepackage{amsmath}
\usepackage{scrextend} 
\allowdisplaybreaks

\let\pa\partial
\let\na\nabla
\let\eps\varepsilon
\newcommand{\diver}{\operatorname{div}}
\newcommand{\R}{{\mathbb R}}
\newcommand{\N}{{\mathbb N}}

\newcommand{\D}{{\mathcal D}}
\newcommand{\E}{{\mathcal E}}
\newcommand{\F}{{\mathcal F}}
\renewcommand{\S}{{\mathcal S}}
\newcommand{\T}{{\mathcal T}}
\newcommand{\V}{{\mathcal V}}
\newcommand{\dist}{{\mathrm{d}}}

\newtheorem{theorem}{Theorem}
\newtheorem{lemma}[theorem]{Lemma}

\newtheorem{remark}{Remark}

\begin{document}

\title[Comparison of a finite-element and finite-volume scheme]{Comparison of a 
finite-element and finite-volume scheme for a degenerate cross-diffusion system
for ion transport}

\author[A. Gerstenmayer]{Anita Gerstenmayer}
\address{Institute for Analysis and Scientific Computing, Vienna University of
	Technology, Wiedner Hauptstra\ss e 8--10, 1040 Wien, Austria}
\email{anita.gerstenmayer@tuwien.ac.at}

\author[A. J\"ungel]{Ansgar J\"ungel}
\address{Institute for Analysis and Scientific Computing, Vienna University of
	Technology, Wiedner Hauptstra\ss e 8--10, 1040 Wien, Austria}
\email{juengel@tuwien.ac.at}

\thanks{The authors acknowledge partial support from   
the Austrian Science Fund (FWF), grants F65, P27352, P30000, and W1245} 

\date{\today}

\begin{abstract}
A structure-preserving implicit Euler finite-element scheme for a degenerate
cross-diffusion system for ion transport is analyzed. The scheme preserves
the nonnegativity and upper bounds of the ion concentrations, the total
relative mass, and it dissipates the entropy (or free energy).
The existence of discrete solutions to the scheme and their convergence towards
a solution to the continuous system is proved. 
Numerical simulations of two-dimensional ion channels using the finite-element scheme 
with linear elements and an alternative finite-volume scheme are presented.
The advantages and drawbacks of both schemes are discussed in detail.
\end{abstract}



\maketitle


\section{Introduction}

Ion channels are pore-forming proteins that create a pathway for charged ions 
to pass through the cell membrane. They are of great biological importance
since they contribute to processes in the nervous system, the coordination
of muscle contraction, and the regulation of secretion of hormones, for instance.
Ion-channel models range from simple systems of differential equations
\cite{HoHu52} as well as Brownian and Langevin dynamics \cite{ISR00,NSS05} to
the widely used Poisson--Nernst--Planck model \cite{Eis98}. The latter model
fails in narrow channels since it neglects the
finite size of the ions. Finite-size interactions can be approximately captured
by adding suitable chemical potential terms \cite{GXWM05,NCE00}, for instance. 
In this paper,
we follow another approach. Starting from a random walk on a lattice, one can
derive in the diffusion limit an extended Poisson--Nernst--Planck model, 
taking into account that ion concentrations might saturate in the narrow channel.
This leads to the appearance of cross-diffusion terms in the evolution equations
for the ion concentrations \cite{BDPS10,SLH09}. These
nonlinear cross-diffusion terms are common in diffusive multicomponent 
systems \cite[Chapter 4]{Jue16}. A lattice-free approach, starting from stochastic
Langevin equations, can be found in \cite{BrCh14}.
The scope of this paper is to present a new finite-element discretization of the
degenerate cross-diffusion system and to compare this scheme to a previously 
proposed finite-volume method \cite{CCGJ18}. 

The dynamics of the ion concentrations $u=(u_1,\ldots,u_n)$ is governed by
the evolution equations
\begin{equation}\label{eq:ion_trans}
  \pa_t u_i + \diver \F_i = 0, \quad \F_i = -D_i\big(u_0\na u_i - u_i\na u_0 
  + u_0u_i\beta z_i\na\Phi\big)\quad\mbox{in }\Omega,\ t>0,
\end{equation}
where $u_{0}=1-\sum_{i=1}^{n}u_i$ denotes the solvent concentration, 
$D_i>0$ is the diffusion constant, $z_i$ the ion charge, and $\beta$ a mobility 
parameter. To be precise, $u_i$ is the mass fraction of the $i$th ion, 
and we refer to $\sum_{i=0}^n u_i=1$ as the total relative mass, just meaning that
the ion-solvent mixture is saturated.
The electric potential $\Phi$ is self-consistently given by 
the Poisson equation
\begin{equation}\label{eq:Phi}
  -\lambda^2\Delta\Phi = \sum_{i=1}^nz_iu_i + f \quad\mbox{in }\Omega,
\end{equation}
with the permanent charge density $f=f(x)$ and the scaled permittivity constant 
$\lambda^2$. The equations are solved in a bounded domain $\Omega\subset\R^d$
with smooth boundary $\pa\Omega$. Equations \eqref{eq:ion_trans} are equipped 
with initial data $u(0)=u^I$ satisfying $0<\sum_{i=1}^n u_i^I<1$. 
The boundary $\pa\Omega$ consists of an insulating part $\Gamma_N$ and the union
$\Gamma_D$ of contacts with external reservoirs:
\begin{align}
  \F_i\cdot\nu=0\quad\mbox{on }\Gamma_N, 
	&\quad u_i=\overline{u}_i\quad\mbox{on }\Gamma_D, \quad i=1,\ldots,n, \label{1.bc1} \\
  \na\Phi\cdot\nu=0\quad\mbox{on }\Gamma_N, 
	&\quad \Phi=\overline{\Phi}\quad\mbox{on }\Gamma_D. \label{1.bc2}
\end{align}

System \eqref{eq:ion_trans}-\eqref{eq:Phi} can be interpreted as a generalized
Poisson--Nernst--Planck model. The usual Poisson--Nernst--Planck equations
\cite{Eis98} follow from \eqref{eq:ion_trans} by setting $u_0=\mbox{const}$. In the
literature, there are several generalized versions of the standard model.
For instance, adding a term involving
the relative velocity differences in the entropy production leads to
cross-diffusion expressions different from \eqref{eq:ion_trans} \cite{HHLLL15}.
This model, however, does not take into account effects from the finite ion size.
Thermodynamically consistent Nernst--Planck models with cross-diffusion terms 
were suggested in \cite{DGM13}, but the coefficients differ from \eqref{eq:ion_trans}.
The model at hand was derived in \cite{BDPS10,SLH09} from a lattice model
taking into account finite-size effects.

Model \eqref{eq:ion_trans}-\eqref{1.bc2} contains some mathematical difficulties.
First, its diffusion matrix $A(u)=(A_{ij}(u))\in\R^{n\times n}$, given by
$A_{ij}(u)=D_iu_i$ for $i\neq j$ and $A_{ii}(u)=D_i(u_0+u_i)$ for $i=1,\ldots,n$
is generally neither symmetric nor positive definite. Second, it degenerates 
in regions where the concentrations vanish. Third, the standard maximum principle
cannot be applied to achieve $0\le u_i\le 1$ for $i=1,\ldots,n$. In the following,
we explain how these issues can be solved.

The first difficulty can be overcome by introducing so-called entropy variables
$w_i$ defined from the entropy (or, more precisely, free energy) of the system,
\begin{equation}\label{1.h}
  H(u) = \int_\Omega h(u)dx, \quad\mbox{where }
  h(u) = \sum_{i=0}^n\int_{\overline{u}_i}^{u_i}\log\frac{s}{\overline{u}_i}ds
  + \frac{\beta\lambda^2}{2}|\na(\Phi-\overline{\Phi})|^2.
\end{equation}
Indeed, writing equations \eqref{eq:ion_trans} in terms of the entropy variables 
$w_1,\ldots,w_n$, given by
\begin{align}
  & \frac{\pa h}{\pa u_i} = w_i - \overline{w}_i, \quad\mbox{where} \label{1.w_i} \\
  & w_i = \log\frac{u_i}{u_0} + \beta z_i\Phi, \quad 
  \overline{w}_i = \log\frac{\overline{u}_i}{\overline{u}_0} + \beta z_i\overline{\Phi},
  \quad i=1,\ldots,n, \nonumber
\end{align}
it follows that
\begin{equation}\label{1.eqw}
  \pa_t u_i(w,\Phi) = \diver\bigg(\sum_{j=1}^n B_{ij}(w,\Phi)\na w_j\bigg),
\end{equation}
where the new diffusion matrix $B=(B_{ij}(w,\Phi))\in\R^{n\times n}$ 
with 
$$
  B_{ij}(w,\Phi) = D_iu_0(w,\Phi)u_i(w,\Phi)\delta_{ij}, \quad i,j=1,\ldots,n,
$$
is symmetric and positive semidefinite (in fact, it is even diagonal). 
This procedure has a thermodynamical background: The quantities $\pa h/\pa u_i$
are known as the chemical potentials, and $B$ is the so-called mobility or
Onsager matrix \cite{DeMa84}.

The transformation to entropy variables also solves the third difficulty.
Solving the transformed system \eqref{1.eqw} for $w=(w_1,\ldots,w_n)$, the
concentrations are given by
\begin{equation}\label{1.uinv}
  u_i(w,\Phi) = \frac{\exp(w_i-\beta z_i\Phi)}{1+\sum_{j=1}^n
	\exp(w_j-\beta z_j\Phi)}, \quad i=1,\ldots,n,
\end{equation}
showing that $u_i$ is positive and bounded from above:
\begin{equation}\label{1.D}
  u(w,\Phi)\in\mathcal{D} := \bigg\{u\in(0,1)^n:\sum_{i=1}^n u_i < 1\bigg\}.
\end{equation}
Moreover, the entropy structure leads to gradient estimates via the entropy inequality
\begin{equation*}
  \frac{dH}{dt} + \frac12\int_\Omega \sum_{i=1}^n D_iu_0u_i|\na w_i|^2 dx \le C,
\end{equation*}
where the constant $C>0$ depends on the Dirichlet boundary data.

Still, we have to deal with the second difficulty, the degeneracy. 
It is reflected in the entropy inequality since we lose the gradient
estimate if $u_i=0$ or $u_0=0$. This problem is overcome by using the ``degenerate''
Aubin-Lions lemma of \cite[Appendix C]{Jue15} or its discrete version in
\cite[Lemma 10]{CCGJ18}. 

These ideas were employed in \cite{BDPS10} for $n=2$ ion species and without electric 
potential to show the global existence of weak solutions. The existence result 
was extended to an arbitrary number of species in \cite{Jue15,ZaJu17}, still
excluding the potential. A global existence result for the full problem
\eqref{eq:ion_trans}-\eqref{1.bc2} was established in \cite{GeJu18}.

We are interested in devising a numerical scheme which preserves
the structure of the continuous system, like nonnegativity, upper bounds,
and the entropy structure, on the discrete level. A first result in this direction
was presented in \cite{CCGJ18}, analyzing a finite-volume scheme preserving
the aforementioned properties. However, the scheme preserves the nonnegativity
and upper bounds only if the diffusion coefficients $D_i$ are all equal, and
the discrete entropy is dissipated only if additionally the potential term
vanishes. In this paper, we propose a finite-element scheme for which the
structure preservation holds under natural conditions. 

Before we proceed, we briefly discuss some related literature. While there are many 
results for the classical Poisson--Nernst--Planck system, see for example 
\cite{LHMZ10,PrSc09}, there seems to be no numerical analysis of the ion-transport 
model \eqref{eq:ion_trans}-\eqref{1.bc2} apart from the finite-volume
scheme in \cite{CCGJ18} and simulations of the stationary equations in \cite{BSW12}. 
Let us mention some other works on 
finite-element methods for related cross-diffusion models. In \cite{BB04}, 
a convergent finite-element scheme for a cross-diffusion population model was 
presented. The approximation is not based on entropy variables, but a regularization 
of the entropy itself that is used to define a regularized system. The same technique 
was employed also in \cite{GS14}. A lumped finite-element method was analyzed in 
\cite{FMSV17} for a reaction-cross-diffusion equation on a stationary surface with 
positive definite diffusion matrix. In \cite{JuLe18}, an implicit Euler Galerkin 
approximation in entropy variables for a Poisson--Maxwell--Stefan system was shown 
to converge. Recently, an abstract framework for the numerical approximation of 
evolution problems with entropy structure was presented in \cite{Egg18}. 
The discretization is based on a discontinuous Galerkin method in time and a 
Galerkin approximation in space. When applied to cross-diffusion systems, this 
approach also leads to an approximation in entropy variables that preserves the 
entropy dissipation. However, neither the existence of discrete solutions nor 
the convergence of the scheme are discussed.

Our main results are as follows:
\begin{itemize}
\item We propose an implicit Euler finite-element scheme for 
\eqref{eq:ion_trans}-\eqref{1.bc2} in entropy variables with linear finite 
elements (Section \ref{sec.fem}). The scheme preserves the
nonnegativity of the concentrations and the upper bounds, the total relative mass,
and it dissipates the discrete entropy associated to \eqref{1.h} if the boundary
data are in thermal equilibrium; see Fthe Remark \ref{rem.fem}.

\item We prove the existence of discrete solutions (Lemma \ref{lem.ex}) and
their convergence to the solution to \eqref{eq:ion_trans}-\eqref{1.bc2} when
the approximation parameters tend to zero (Theorem \ref{thm.conv}).
The convergence rate can be only computed numerically and is approximately
of second order (with respect to the $L^2$ norm).

\item The finite-element scheme and the finite-volume scheme of \cite{CCGJ18}
(recalled in Section \ref{sec.fvm}) 
are applied to two test cases in two space dimensions: 
a calcium-selective ion channel and a bipolar ion channel (Section \ref{sec:Num}).
Static current-voltage curves show the rectifying behavior of the bipolar
ion channel.

\item The advantages and drawbacks of both schemes are discussed 
(Section \ref{sec.Con}). The finite-element scheme allows for structure-preserving
properties under natural assumptions, while the finite-volume scheme can be
analyzed only under restrictive conditions. On the other hand, the
finite-volume scheme allows for vanishing initial concentrations and faster
algorithms compared to the finite-element scheme due to the highly nonlinear
structure of the latter formulation. 
\end{itemize}


\section{The finite-element scheme}\label{sec.fem}

\subsection{Notation and assumptions}

Before we define the finite-element discretization, we introduce our notation
and make precise the conditions assumed throughout this section. We assume:

\begin{labeling}{(A44)}
\item[(H1)] Domain: $\Omega\subset\R^d$ ($d=2$ or $d=3$) is an open, bounded, 
polygonal domain with $\pa\Omega=\Gamma_D\cup\Gamma_N\in C^{0,1}$, 
$\Gamma_D\cap\Gamma_N=\emptyset$, $\Gamma_N$ is open in $\pa\Omega$, and
$\mbox{meas}(\Gamma_D)>0$.

\item[(H2)] Parameters: $T>0$, $D_i>0$, $\beta>0$, and $z_i\in\R$,
$i=1,\ldots,n$.

\item[(H3)] Background charge: $f\in L^\infty(\Omega)$.

\item[(H4)] Initial and boundary data: $u_i^{I}\in H^2(\Omega)$ and
$\overline{u}_i\in H^2(\Omega)$ satisfy $u_i^{I} > 0$, $\overline{u}_i > 0$
for $i=1,\ldots,n$,  
$1-\sum_{i=1}^n u_i^{I} > 0$, $1-\sum_{i=1}^n \overline{u}_i > 0$
in $\Omega$, and $\overline{\Phi}\in H^2(\Omega)\cap L^\infty(\Omega)$.
\end{labeling}

The $H^2$ regularity of the initial and boundary data ensures that the standard 
interpolation converges to the given data, see \eqref{interp} below.

We consider equations \eqref{eq:ion_trans} on a finite time interval $(0,T)$ 
with $T>0$. For simplicity, we use a uniform time discretization with time step 
$\tau>0$ and set $t^k=k\tau$ for $k=1,\ldots,N$, where $N\in \N$ is given and 
$\tau=T/N$.

For the space discretization, we introduce a family $\T_h$ ($h>0$)
of triangulations of $\Omega$, consisting of open polygonal convex subsets of $\Omega$
(the so-called cells)
such that $\overline{\Omega}=\cup_{K\in\T_h}\overline{K}$ with maximal diameter 
$h=\max_{K\in\T_h}\text{diam}(K)$.  We assume that the corresponding family of edges
$\E$ can be split into internal and external edges, $\E=\E_{\rm int}\cup\E_{\rm ext}$
with $\E_{\rm int} = \{\sigma\in\E:\sigma\subset\Omega\}$ and $\E_{\rm ext}=\{\sigma
\in\E:\sigma\subset\pa\Omega\}$. Each exterior edge is assumed to be an element of
either the Dirichlet or Neumann boundary, i.e.\ $\E_{\rm ext}=\E_{\rm ext}^D\cup
\E_{\rm ext}^N$. 
For given $K\in\T_h$, we define the set $\E_K$ of the edges of $K$, which is the union
of internal edges and edges on the Dirichlet or Neumann boundary, and we set 
$\E_{K,\rm int}=\E_K\cap \E_{\rm int}$.

In the finite-element setting, the triangulation is completed by the set of nodes
$\{p_j:j\in J\}$. We impose the following regularity assumption on the
mesh. There exists a constant $\gamma\ge 1$ such that
$$
  \rho_K \le h_K \le \gamma\rho_K \quad\mbox{for all }K\in\T_h,
$$
where $\rho_K$ is the radius of the incircle and $h_K$ is the diameter of $K$.

We associate with $\T_h$ the usual conforming finite-element spaces
\begin{align*}
  \S(\T_h) &:= \{ \xi\in C^0(\overline{\Omega}): \xi|_K \text{ is linear for all }
	K\in\T_h \}\subset H^1(\Omega), \\
  \S_D(\T_h) &:= \S(\T_h)\cap H^1_D(\Omega),
\end{align*}
and $H_D^1(\Omega)$ is the set of $H^1(\Omega)$ functions that vanish on $\Gamma_D$
in the weak sense. 
Let $\{\chi_j\}_{j\in J}$ be the standard basis functions for $\S(\T_h)$ with 
$\chi_j(p_i)=\delta_{ij}$ for all $i,j\in J$. We define the nodal interpolation 
operator $I_h:C^0(\overline{\Omega})\to \S(\T_h)$ via $(I_h v)(p_j)=v(p_j)$ for all 
$v\in \S(\T_h)$ and $j\in J$. 
Due to the regularity assumptions on the mesh, $I_h$ has the following 
approximation property (see, e.g., \cite[Chapter 3]{Cia02}): 
\begin{equation}\label{interp}
  \lim_{h\to 0}\|\phi-I_h\phi\|_{H^1(\Omega)} = 0 
	\quad \text{for all }\phi\in H^2(\Omega).
\end{equation}


\subsection{Definition of the scheme}

To define the finite-element scheme, we need to approximate the initial and boundary
data. We set $w^0_i=I_h(\log(u_i^I/u_0^I))+\beta z_i\Phi^0$, where $\Phi^0$ is the 
standard finite-element solution to the linear equation \eqref{eq:Phi} with $u_i^I$ 
on the right-hand side. Furthermore, we set 
$\overline{w}_h=I_h(\log(\overline{u}_i/\overline{u}_0)+\beta z_i\overline{\Phi})$ 
and $\overline{\Phi}_h=I_h(\overline{\Phi})$.

The finite-element scheme is now defined as follows. Given $w^{k-1}\in\S(\T_h)^n$ and 
$\Phi^{k-1}\in\S(\T_h)$, find $w^k-\overline{w}_h\in \S_D(\T_h)^n$ and
$\Phi^k-\overline{\Phi}_h\in\S_D(\T_h)$ such that
\begin{align}
  & \frac{1}{\tau}\int_\Omega\big(u(w^k,\Phi^k) - u(w^{k-1},\Phi^{k-1})
	\big)\cdot\phi\, dx \nonumber \\
	&\phantom{xxxx}{}+ \int_\Omega\na\phi:B(w^k,\Phi^k)\na w^kdx 
	+ \eps\int_\Omega (w^k-\overline{w}_h)\cdot\phi\,dx = 0, \label{scheme1} \\
	& \lambda^2\int_\Omega\na\Phi^k\cdot\na\theta dx = \int_\Omega\bigg(\sum_{i=1}^n
	z_i u_i(w^k,\Phi^k)+f\bigg)\theta dx \label{scheme2}
\end{align}
for all $\phi\in \S_D(\T_h)^n$ and $\theta\in \S_D(\T_h)$. 
The symbol ``:'' signifies the Frobenius matrix product; here, the expression
reduces to 
$$
  \na\phi:B(w^k,\Phi^k)\na w^k = \sum_{i=1}^n D_iu_i(w^k,\Phi^k)u_0(w^k,\Phi^k)
	\na\phi_i\cdot\na w_i^k.
$$
The term involving the parameter
$\eps>0$ is only needed to guarantee the coercivity of \eqref{scheme1}-\eqref{scheme2}. 
Indeed, the diffusion matrix $B(w^k,\Phi^k)$ degenerates when $w^k_i\to -\infty$,
and the corresponding bilinear form is only positive semidefinite. 
To emphasize the dependence on the mesh 
and $\eps$, we should rather write $w^{(h,\eps,k)}$ instead of $w^k$ and similarly for 
$\Phi^k$; however, for the sake of presentation, we will mostly omit the additional 
superscripts. The original variables are recovered by computing
$u^{k}=u(w^{k},\Phi^{k})$ according to \eqref{1.uinv}. 
Setting $u^{(\tau)}(x,t)=u^k(x)$ for $x\in\Omega$, $t\in((k-1)\tau,k\tau]$,
$k=1,\ldots,N$, and $u^{(\tau)}(\cdot,0)=I_hu^I$ as well as similarly for
$\Phi^{(\tau)}$, we obtain piecewise constant in time functions.


\subsection{Existence of discrete solutions}

The first result concerns the existence of solutions to the nonlinear finite-element 
scheme \eqref{scheme1}-\eqref{scheme2}. 

\begin{lemma}[Existence of solutions and discrete entropy inequality]\label{lem.ex}
There exists a so\-lu\-tion to scheme \eqref{scheme1}-\eqref{scheme2} that satisfies 
the following discrete entropy inequality:
\begin{equation}
  H(u^k) + \tau\int_\Omega\na (w^k-\overline{w}_h):B(w^k,\Phi^k)\na w^k dx 
	+ \eps\tau \|w^k-\overline{w}_h\|_{L^2(\Omega)}^2 \le H(u^{k-1}), \label{2.epi}
\end{equation}
where $H$ is defined in \eqref{1.h} and 
$u^k=u(w^k,\Phi^k)$, $u^{k-1}=u(w^{k-1},\Phi^{k-1})$ are defined in \eqref{1.uinv}.
\end{lemma}

The proof of the lemma is similar to the proof of Theorem 1 in \cite{GeJu18}.
The main difference is that in \cite{GeJu18}, a regularization term of the type
$\eps((-\Delta)^m w^k+w^k)$ has been added to achieve via 
$H^m(\Omega)\hookrightarrow L^\infty(\Omega)$ for $m>d/2$
compactness and $L^\infty$ solutions.
In the finite-dimensional setting, this embedding is not necessary but we still
need the regularization $\eps w^k$ to conclude coercivity. 
We conjecture that this regularization is just technical but currently, we are not
able to remove it. Note, however, that we can use arbitrarily small values of $\eps$ 
in the numerical simulations such that the additional term does not affect the
solution {\em practically}.

\begin{proof}[Proof of Lemma \ref{lem.ex}.]
Let $y\in \S(\T_h)^n$ and $\delta\in[0,1]$. 
There exists a unique solution $\Phi^k$ to \eqref{scheme2} with $w^k$ replaced by 
$y+\overline{w}_h$,
satisfying $\Phi^k-\overline\Phi_h\in\S_D(\T_h)$, since the function
$\Phi\mapsto u_i(y,\Phi)$ is bounded. Moreover, the estimate
\begin{equation}\label{2.Phi}
  \|\Phi^k\|_{H^1(\Omega)} \le C(1+\|\overline{\Phi}_h\|_{H^1(\Omega)})
\end{equation}
holds for some constant $C>0$. 

Next, we consider the linear problem
\begin{equation}\label{2.LM}
  a(v,\phi)=F(\phi)\quad\mbox{for all }\phi\in \S_D(\T_h)^n,
\end{equation}
where
\begin{align*}
  a(v,\phi) &= \int_\Omega\na\phi:B(y+\overline{w}_h,\Phi^k)\na v\, dx
	+ \eps\int_\Omega v\cdot \phi \,dx, \\
	F(\phi) &= -\frac{\delta}{\tau}\int_\Omega\big(u(y+\overline{w}_h,\Phi^k)
	- u(w^{k-1},\Phi^{k-1})\big)\cdot\phi dx \\
	&\phantom{xx}{}
	- \delta\int_\Omega\na\phi:B(y+\overline{w}_h,\Phi^k)\na\overline{w}_h dx.
\end{align*}
The bilinear form $a$ and the linear form $F$ are continuous on $\S_D(\T_h)^n$.
The equivalence of all norms on the finite-dimensional space $\S_D(\T_h)$
implies the coercivity of $a$,
$$
  a(v,v) \ge \eps\|v\|_{L^2(\Omega)}^2 \ge \eps C\|v\|_{H^1(\Omega)}^2.
$$
By the Lax-Milgram lemma, there exists a unique solution $v\in \S_D(\T_h)^n$ to
this problem. This defines the fixed-point operator 
$S:\S_D(\T_h)^n\times[0,1]\to\S_D(\T_h)^n$, $S(y,\delta)=v$. The inequality
$$
  \eps C\|v\|_{H^1(\Omega)}^2 \le a(v,v) = F(v) \le C(\tau)\|v\|_{H^1(\Omega)}
$$
shows that all elements $v$ are bounded independently of $y$ and $\delta$
and thus, all fixed points $v=S(v,\delta)$ are uniformly bounded.
Furthermore, $S(y,0)=0$ for all $y\in \S_D(\T_h)^n$. The continuity of $S$
follows from standard arguments and the compactness comes from the fact that
$\S_D(\T_h)^n$ is finite-dimensional. By the Leray-Schauder fixed-point theorem,
there exists $v^k\in \S_D(\T_h)^n$ such that $S(v^k,1)=v^k$, and 
$w^k:=v^k+\overline{w}_h$ is a solution to \eqref{scheme1}.

The discrete entropy inequality \eqref{2.epi} is proven by using
$\tau(w^k-\overline{w}_h)\in \S_D(\T_h)^n$ as a test function in \eqref{scheme1}
and exploiting the convexity of $H$,
$$
  \int_\Omega(u^k-u^{k-1})\cdot (w^k-\overline{w}_h)dx
	= \int_\Omega(u^k-u^{k-1})\cdot \na h(u^k)dx
	\ge H(u^k)-H(u^{k-1}),
$$
which concludes the proof.
\end{proof}

\begin{remark}[Structure-preservation of the scheme]\label{rem.fem}\rm
Lemma \ref{lem.ex} shows that if the boundary data is in thermal equilibrium, 
i.e.\ $\na \overline{w}_h=0$, then
the finite-ele\-ment scheme \eqref{scheme1}-\eqref{scheme2}
dissipates the entropy \eqref{1.h}, i.e.\ $H(u^k)\le H(u^{k-1})$. Moreover, 
it preserves the invariant region $\mathcal{D}$, i.e.\ $u^k\in\mathcal{D}$, and
the mass fraction $u_i^k$ is nonnegative and bounded by one. 
The scheme conserves the total relative mass, i.e.\
$\sum_{i=0}^n \|u_i^k\|_{L^1(\Omega)}=1$, 
which is a direct consequence of the definition of $u_0^k$.
\qed
\end{remark}


\subsection{Uniform estimates}

The next step is the derivation of a priori estimates uniform in the parameters
$\eps$, $\tau$, and $h$. To this end, we transform back to the original variable
$u^k$ and exploit the discrete entropy inequality \eqref{2.epi}.

\begin{lemma}[A priori estimates]\label{lem.est}
For the solution to the finite-element scheme from Lemma \ref{lem.ex},
 the following estimates hold:
\begin{align}
  \|u_i^k\|_{L^\infty(\Omega)} 
	+ \eps\tau\sum_{j=1}^k\|w_i^j-\overline{w}_{i,h}\|_{L^2(\Omega)}^2
	&\le C, \label{2.est1} \\
	\tau\sum_{j=1}^k\Big(\|(u_0^j)^{1/2}\|_{H^1(\Omega)}^2 + \|u_0^j\|_{H^1(\Omega)}^2 
	+ \|(u_0^j)^{1/2}\na(u_i^j)^{1/2}\|_{L^2(\Omega)}^2\Big) &\le C, \label{2.est2} 
\end{align}
for $i=1,\ldots,n$, where here and in the following, 
$C>0$ is a generic constant independent of $\eps$, $\tau$, and $h$. 
\end{lemma}

\begin{proof}
As the proof is similar to that one in the continuous setting, we give only a sketch.
Observe that the definition of the entropy variables implies that
$0<u_i^k<1$ in $\Omega$ for $i=1,\ldots,n$ and $k=1,\ldots,N$. It is shown in
the proof of Lemma 6 of \cite{GeJu18} that
\begin{align*}
  \na (w^k-\overline{w}_h):B(w^k,\Phi^k)\na w^k
	&\ge \frac{D_{\rm min}}{4}\sum_{i=1}^n u_i^k u_0^k
	\bigg|\na\log\frac{u_i^k}{u_0^k}\bigg|^2 
	- \frac{D_{\rm min}}{2}\sum_{i=1}^n|\beta z_i\na\Phi^k|^2 \\
	&\phantom{xx}{}- \frac{D_{\rm max}}{2}\sum_{i=1}^n|\na \overline{w}_i|^2,
\end{align*}
where $D_{\rm min}=\min_{i=1,\ldots,n}D_i$ and $D_{\rm max}=\max_{i=1,\ldots,n}D_i$.
Then \eqref{2.epi} gives
\begin{align*}
  H(u^k) &+ \tau\frac{D_{\rm min}}{4}\int_\Omega\sum_{i=1}^n u_i^k u_0^k
	\bigg|\na\log\frac{u_i^k}{u_0^k}\bigg|^2 dx
	+ \eps\tau\|w^k-\overline{w}_h\|_{L^2(\Omega)}^2 \\
	&\le H(u^{k-1}) + \tau\frac{D_{\rm min}}{2}\sum_{i=1}^n|\beta z_i\na\Phi^k|^2 dx
	+ \tau\frac{D_{\rm max}}{2}\int_\Omega\sum_{i=1}^n|\na \overline{w}_{i,h}|^2 dx.
\end{align*}
We resolve this recursion to find that
\begin{align*}
  H(u^k) &+ \tau\frac{D_{\rm min}}{4}\sum_{j=1}^k
	\int_\Omega\sum_{i=1}^n u_i^ju_0^j
	\bigg|\na\log\frac{u_i^j}{u_0^j}\bigg|^2 dx
	+ \eps\tau\sum_{j=1}^k \|w^j-\overline{w}_{h}\|_{L^2(\Omega)}^2 \\
	&\le H(u^0) + \tau\frac{D_{\rm min}}{2}\sum_{j=1}^k\int_\Omega\sum_{i=1}^n
	|\beta z_i\na\Phi^j|^2 dx 
	+ \tau k\frac{D_{\rm max}}{2}\int_\Omega\sum_{i=1}^n|\na \overline{w}_{i,h}|^2 dx.
\end{align*}
The right-hand side is bounded because of \eqref{2.Phi}, $\tau k\le T$, and the
boundedness of the interpolation operator. Inserting the identity
$$
  \sum_{i=1}^n u_i^ju_0^j\bigg|\na\log\frac{u_i^j}{u_0^j}\bigg|^2
= 4u_0^j\sum_{i=1}^n|\na(u_i^j)^{1/2}|^2 + |\na u_0^j|^2 + 4|\na(u_0^j)^{1/2}|^2,
$$
the estimates follow.
\end{proof}


\subsection{Convergence of the scheme}

The a priori estimates from the previous lemma allow us to formulate our main result, 
the convergence of the finite-element solutions to a solution to the
continuous model \eqref{eq:ion_trans}-\eqref{1.bc2}.

\begin{theorem}[Convergence of the finite-element solution]\label{thm.conv}
Let $(u^{(h,\eps,\tau)},\Phi^{(h,\eps,\tau)})$ be an approximate solution 
constructed from scheme \eqref{scheme1}-\eqref{scheme2}. Set 
$u^{(h,\eps,\tau)}_0=1-\sum_iu^{(h,\eps,\tau)}_i$.
Then there exist functions $u_0$, $u=(u_1,\ldots,u_n)$, and $\Phi$,
satisfying $u(x,t)\in\overline\D$ ($\D$ is defined in \eqref{1.D}), 
$u_0=1-\sum_{i=1}^n u_i$ in $\Omega$, the regularity
$$
  u_0^{1/2},\, u_0^{1/2}u_i,\, \Phi \in L^2(0,T;H^1(\Omega)),\quad
	\pa_tu_i \in L^2(0,T;H^1_D(\Omega)')
$$ 
for $i=1,\ldots,n$, such that as $(h,\eps,\tau)\to 0$,
\begin{align*}
  (u^{(h,\eps,\tau)}_{0})^{1/2}\to u_0^{1/2},\ 
	(u^{(h,\eps,\tau)}_{0})^{1/2}u_{i}^{(h,\eps,\tau)}\to u_0^{1/2}u_i
	&\quad\mbox{strongly in }L^2(\Omega\times(0,T)), \\
  \Phi^{(h,\eps,\tau)}\to \Phi &\quad\mbox{strongly in }L^2(\Omega\times(0,T)),
\end{align*}
and $(u,\Phi)$ are a weak solution to
\eqref{eq:ion_trans}-\eqref{1.bc2}. In particular, for all 
$\phi\in L^2(0,T;H^1_D(\Omega))$ and $i=1,\ldots,n$, it holds that
\begin{align}
  &\int_0^T\langle\pa_t u_i,\phi\rangle\, dt + D_i\int_0^T\int_\Omega
	u_0^{1/2}\big(\na(u_0^{1/2}u_i) - 3u_i\na u_0^{1/2}\big)\cdot\na\phi\, dxdt 
	\nonumber \\
	&\phantom{xx}{}
	+ D_i\int_0^T\int_\Omega u_i u_0\beta z_i\na\Phi\cdot\na\phi\, dxdt = 0, 
	\label{1.weak1} \\
	& \lambda^2\int_0^T\int_\Omega\na\Phi\cdot\na\phi\, dxdt
	= \int_0^T\int_\Omega\bigg(\sum_{i=1}^n z_iu_i + f\bigg)\phi\, dxdt, \label{1.weak2}
\end{align}
where $\langle\cdot,\cdot\rangle$ is the duality pairing in $H^1_D(\Omega)'$
and $H_D^1(\Omega)$, 
and the boundary and initial conditions are satisfied in a weak sense.
\end{theorem}

\begin{proof}
We pass first to the limit $(\eps,h)\to 0$ and then $\tau\to 0$, since the latter limit
can be performed as in the proof of Theorem 1 in \cite{GeJu18}. 
Fix $k\in\{1,\ldots,N\}$ and let $u_i^{(\eps,h)}=u_i^{(\eps,h,k)}$ and 
$\Phi^{(\eps,h)}=\Phi^{(\eps,h,k)}$ be the approximate solution from Lemma \ref{lem.ex}. 
We set $u_0^{(\eps,h)}=1-\sum_{i=1}^n u_i^{(\eps,h)}$. 
Using the compact embedding $H^1(\Omega)\hookrightarrow L^2(\Omega)$ and 
the a priori estimates from Lemma \ref{lem.est}, it follows that there exists a
subsequence which is not relabeled such that, as $(\eps,h)\to 0$,
\begin{align}
  u_i^{(\eps,h)} \rightharpoonup^* u_i^k &\quad\mbox{weakly* in }L^\infty(\Omega),
	\ i=1,\ldots,n,	\label{2.eps0} \\
  (u_0^{(\eps,h)})^{1/2}\rightharpoonup (u_0^k)^{1/2}, 
	\quad \Phi^{(\eps,h)}\rightharpoonup\Phi^k
	&\quad\mbox{weakly in }H^1(\Omega), \label{2.H1} \\
	u_0^{(\eps,h)}\to u_0^k, \quad \Phi^{(\eps,h)}\to \Phi^k
	&\quad\mbox{strongly in }L^2(\Omega), \label{2.eps1} \\
	\eps (w_i^{(\eps,h)}-\overline{w}_{i,h})\to 0 &\quad\mbox{strongly in }L^2(\Omega). 
	\label{2.weps}
\end{align}
Combining \eqref{2.est2} and \eqref{2.H1}, we infer that (up to a subsequence)
\begin{equation}\label{2.comb}
  u_i^{(\eps,h)}(u_0^{(\eps,h)})^{1/2}\rightharpoonup u_i^k(u_0^k)^{1/2}
  \quad\mbox{weakly in }H^1(\Omega)\mbox{ and strongly in }L^2(\Omega).
\end{equation}

Next, let $\phi\in (H^2(\Omega)\cap H_D^1(\Omega))^n$. 
As we cannot use $\phi_i$ directly as a 
test function in \eqref{scheme1}, we take $I_h\phi\in\S_D(\T_h)^n$, where
$I_h$ is the interpolation operator, see \eqref{interp}. 
In order to pass to the limit in \eqref{scheme1}, we rewrite the integral involving
the diffusion matrix:
\begin{align}
  \int_\Omega & \na (I_h\phi):B(w^{(\eps,h)},\Phi^{(\eps,h)})\na w^{(\eps,h)} dx
	= \int_\Omega\sum_{i=1}^n D_i u_i^{(\eps,h)}u_0^{(\eps,h)}\na w_i^{(\eps,h)}
	\cdot\na (I_h\phi_i) dx \nonumber\\
	&= \int_\Omega\sum_{i=1}^n D_i\Big((u_0^{(\eps,h)})^{1/2}\na\big(u_i^{(\eps,h)}
	(u_0^{(\eps,h)})^{1/2}\big) 
	- 3u_i^{(\eps,h)}(u_0^{(\eps,h)})^{1/2}\na(u_0^{(\eps,h)})^{1/2} \nonumber\\ 
	&\phantom{xx}{}
	+ \beta z_i u_i^{(\eps,h)} u_0^{(\eps,h)}\na\Phi^{(\eps,h)}\Big)
	\cdot\na( I_h\phi_i) dx. \label{for_limit}
\end{align}
We estimate each of the above summands separately. 
For the last term, we proceed as follows:
\begin{align}
  \bigg| \int_{\Omega} & u_i^{(\eps,h)} u_0^{(\eps,h)}\na\Phi^{(\eps,h)}
	\cdot\na(I_h\phi_i)dx - \int_{\Omega}u_i^k u_0^k\na\Phi^k\cdot\na\phi_idx\bigg| 
	\nonumber \\
  &\le \bigg| \int_{\Omega}u_i^{(\eps,h)} u_0^{(\eps,h)}\na\Phi^{(\eps,h)}
	\cdot\na( I_h\phi_i -\phi_i)dx \bigg| \nonumber \\
	&\phantom{xx}{}+ \bigg| \int_{\Omega}(u_i^{(\eps,h)} u_0^{(\eps,h)}\na\Phi^{(\eps,h)}
	-u_i^k u_0^k\na\Phi^k)\cdot\na\phi_idx \bigg| \nonumber \\
  &\le \|u_i^{(\eps,h)} u_0^{(\eps,h)}\na\Phi^{(\eps,h)}\|_{L^2(\Omega)}
	\|\na( I_h\phi_i -\phi_i)\|_{L^2(\Omega)} \nonumber \\
	&\phantom{xx}{}+ \bigg| \int_{\Omega}(u_i^{(\eps,h)} u_0^{(\eps,h)}\na\Phi^{(\eps,h)}
	-u_i^k u_0^k\na\Phi^k)\cdot\na\phi_idx \bigg|. \label{conv2}
\end{align}
Similarly as for \eqref{2.comb}, it follows that
$$
  u_i^{(\eps,h)}u_0^{(\eps,h)}\to u_i^ku_0^k \quad\mbox{strongly in }L^2(\Omega).
$$
Then, together with the weak convergence of $\na\Phi^{(\eps,h)}$, we infer that 
$$
  u_i^{(\eps,h)}u_0^{(\eps,h)}\na\Phi^{(\eps,h)}
	\rightharpoonup u_i^ku_0^k\na\Phi^k \quad\mbox{weakly in }L^1(\Omega).
$$
Since $(u_i^{(\eps,h)}u_0^{(\eps,h)}\na\Phi^{(\eps,h)})$ is bounded in $L^2(\Omega)$,
this weak convergence also holds in $L^2(\Omega)$. 
Because of the interpolation property \eqref{interp} and estimate \eqref{conv2},
$$
  \int_{\Omega}u_i^{(\eps,h)} u_0^{(\eps,h)}\na\Phi^{(\eps,h)}
	\cdot\na( I_h\phi_i)dx 
	\to \int_{\Omega}u_i^k u_0^k\na\Phi^k\cdot\na\phi_idx.
$$
Following the arguments of Step 3 in \cite[Section 2]{GeJu18} (using \eqref{2.comb}), 
we have
\begin{align*}
  (u_0^{(\eps,h)})^{1/2}&\na\big(u_i^{(\eps,h)}(u_0^{(\eps,h)})^{1/2}\big) 
	- 3u_i^{(\eps,h)}(u_0^{(\eps,h)})^{1/2}\na\big((u_0^{(\eps,h)})^{1/2}\big) \\
  &\rightharpoonup (u_0^k)^{1/2}\na\big(u_0^k(u_0^k)^{1/2}\big)
	- 3u_i^k(u_0^k)^{1/2}\na((u_0^k)^{1/2}) \quad\mbox{weakly in }L^2(\Omega).
\end{align*}
Thus, the limit $(\eps,h)\to 0$ in \eqref{for_limit} gives
\begin{align*}
  \lim_{(\eps,h)\to 0}\int_\Omega &
	\na (I_h\phi):B(w^{(\eps,h)},\Phi^{(\eps,h)})\na w^{(\eps,h)} dx 
	= \int_\Omega\sum_{i=1}^n D_i\Big((u_0^k)^{1/2}\na\big(u_0^k(u_0^k)^{1/2}\big) \\
	&\phantom{xx}{}
	- 3u_i^k(u_0^k)^{1/2}\na((u_0^k)^{1/2}) + \beta z_iu_i^ku_0^k\na\Phi^k\Big)\cdot
	\na\phi_i dx.
\end{align*}
Furthermore, we deduce from \eqref{2.weps} that
$$
  \bigg|\eps\int_\Omega(w_i^{(\eps,h)}-\overline{w}_{i,h})(I_h\phi_i)dx\bigg|
	\le \eps\|w_i^{(\eps,h)}-\overline{w}_{i,h}\|_{L^2(\Omega)}
	\|I_h\phi_i\|_{L^2(\Omega)} \to 0.
$$
Thus, passing to the limit $(\eps,h)\to 0$ in scheme \eqref{scheme1}-\eqref{scheme2}
leads to
\begin{align*}
  &\frac{1}{\tau}\int_\Omega(u^k-u^{k-1})\cdot\phi dx 
  + \int_\Omega\sum_{i=1}^n D_i(u_0^k)^{1/2}
  \big(\na(u_i^k(u_0^k)^{1/2}) - 3u_i^k\na (u_0^k)^{1/2}\big)\cdot\na\phi_i dx \\
  &\phantom{xx}{}
  + \int_\Omega\sum_{i=1}^n D_i\beta z_i u_i^ku_0^k\na\Phi^k
  \cdot\na \phi_i dx = 0,	\\
  & \lambda^2\int_\Omega\na\Phi^k\cdot\na\theta dx 
  = \int_\Omega\bigg(\sum_{i=1}^n z_iu_i^k + f\bigg)\theta dx,
\end{align*}
for all $\phi_i$, $\theta\in H^2(\Omega)\cap H_D^1(\Omega)$. A density argument shows 
that we can take test functions $\phi_i$, $\theta\in H_D^1(\Omega)$. 
The a priori estimates from Lemma \ref{lem.est} remain valid in the weak limit. 

Now the limit $\tau\to 0$ can be done exactly as in \cite[Theorem 1, Step 4]{GeJu18}, 
which concludes the proof.
\end{proof}


\section{The finite-volume scheme}\label{sec.fvm}

We briefly recall the finite-volume scheme from \cite{CCGJ18} and summarize
the assumptions and results, as this is necessary for the comparison of the
finite-element and finite-volume scheme in Section \ref{sec.Con}.

We assume that Hypotheses (H1)-(H4) from the previous section hold and we use
the same notation for the time and space discretizations. For a two-point
approximation of the discrete gradients, we require additionally that the
mesh is admissible in the sense of \cite[Definition 9.1]{EGH00}.
This means that a family of points $(x_K)_{K\in\T}$ is associated to the cells 
and that the line connecting the points $x_K$ and $x_L$ of two neighboring cells 
$K$ and $L$ is perpendicular to the edge $K|L$. For $\sigma\in\E_{\rm int}$ with 
$\sigma=K|L$, we denote by $\dist_\sigma=\dist(x_K,x_L)$ 
the Euclidean distance between $x_K$ and $x_L$, while
for $\sigma\in\E_{\rm ext}$, we set $\dist_\sigma=\dist(x_K,\sigma)$.
For a given edge $\sigma\in\E$, the transmissibility coefficient is defined by
$\tau_\sigma = \text{m}(\sigma)/\dist_\sigma$,
where $\text{m}(\sigma)$ denotes the Lebesgue measure of $\sigma$.

For the definition of the scheme, we approximate the initial, boundary, and given 
functions on the elements $K\in\T$ and edges $\sigma\in\E$:
\begin{align*}
  u^{I}_{i,K} &= \frac{1}{\text{m}(K)}\int_{K}u^{I}_i(x)dx, &
	f_K &= \frac{1}{\text{m}(K)}\int_{K}f(x)dx, \\
  \overline{u}_{i,\sigma} &= \frac{1}{\text{m}(\sigma)}\int_{\sigma}\overline{u}_ids, & 
	\overline{\Phi}_{\sigma} &= \frac{1}{\text{m}(\sigma)}\int_{\sigma}\overline{\Phi}ds,
\end{align*}
and we set $u_{0,K}^I= 1-\sum_{i=1}^n u_{i,K}^I$ and 
$\overline{u}_{0,\sigma}=1-\sum_{i=1}^n\overline{u}_{i,\sigma}$. 
Furthermore, we introduce the discrete gradients
\begin{align}
	&\text{D}_{K,\sigma}(u_i) = u_{i,K,\sigma}-u_{i,K}, 
  &\mbox{where}\quad u_{i,K,\sigma} = \begin{cases}
	u_{i,L} \quad&\text{ for }\sigma\in\E_{\rm int},\ \sigma=K|L,\\
  \overline{u}_{i,\sigma} \quad&\text{ for }\sigma\in\E^D_{{\rm ext},K},\\
  u_{i,K} &\text{ for }\sigma\in\E^N_{{\rm ext},K}. 
  \end{cases} \nonumber
\end{align}

The numerical scheme is now defined as follows. Let $K\in\T$, $k\in\{1,\ldots,N\}$, 
$i\in\{1,\ldots,n\}$, and $u_{i,K}^{k-1}\ge 0$ be given. Then the values $u_{i,K}^k$ 
are determined by the implicit Euler scheme
\begin{equation}\label{2.equ}
  \text{m}(K)\frac{u_{i,K}^k-u_{i,K}^{k-1}}{\Delta t} 
	+ \sum_{\sigma\in\E_K}\F_{i,K,\sigma}^k = 0,
\end{equation}
where the fluxes $\F_{i,K,\sigma}^k$ are given by the upwind scheme
\begin{equation*}
  \F_{i,K,\sigma}^k = -\tau_\sigma D_i\Big(u_{0,\sigma}^k\text{D}_{K,\sigma}(u_i^k)
	- u_{i,\sigma}^k\big(\text{D}_{K,\sigma}(u_0^k) - \widehat{u}_{0,\sigma,i}^k
	\beta z_i\text{D}_{K,\sigma}(\Phi^k)\big)\Big).
\end{equation*}
Here, we have set
\begin{align*}
  & u_{0,K}^k=1-\sum_{i=1}^n u_{i,K}^k, \quad 
	u_{0,\sigma}^k = \max\{u_{0,K}^k,u_{0,L}^k\}, \\
	& u^k_{i,\sigma} = \begin{cases}
  u^k_{i,K} \quad & \text{if } \V^k_{i,K,\sigma}\ge 0, \\
  u^k_{i,K,\sigma} & \text{if }\V^k_{i,K,\sigma}< 0, 
  \end{cases}, \quad
  \widehat{u}_{0,\sigma,i}^k = \begin{cases}
  u^k_{0,K} \quad & \text{if }z_i\text{D}_{K,\sigma}(\Phi^k)\ge 0, \\
  u^k_{0,K,\sigma} & \text{if }z_i\text{D}_{K,\sigma}(\Phi^k)< 0, 
  \end{cases},
\end{align*}
and $\V_{i,K,\sigma}^k$ is the ``drift part'' of the flux,
\begin{equation*}
  \V_{i,K,\sigma}^k = \text{D}_{K,\sigma}(u_0^k) - \widehat{u}_{0,\sigma,i}^k
	\beta z_i\text{D}_{K,\sigma}(\Phi^k)
\end{equation*}
for $i=1,\ldots,n$. Observe that we employed a double upwinding: one related to
the electric potential, defining $\widehat{u}_{0,\sigma,i}^k$, and
another one related to the drift part of the flux, $\V_{i,K,\sigma}^k$.
The potential is computed via
\begin{equation*}
  -\lambda^2\sum_{\sigma\in\E_K}\tau_\sigma\text{D}_{K,\sigma}(\Phi^k)
	= \text{m}(K)\bigg(\sum_{i=1}^n z_iu_{i,K}^k + f_K\bigg).
\end{equation*}

The finite volume scheme preserves the structure of the continuous equations only 
under certain assumptions:
\begin{labeling}{(A44)}
	\item[(A1)] $\pa\Omega=\Gamma_N$, i.e., we impose no-flux
	boundary conditions on the whole boundary.
	\item[(A2)] The diffusion constants are equal, $D_i=D>0$ for $i=1,\ldots,n$.
	\item[(A3)] The drift terms are set to zero, $\Phi\equiv 0$.
\end{labeling}
Without these assumptions, we can only assure the nonnegativity of the discrete
concentrations $u_i$, $i=1,\ldots,n$. Since we lack a maximum principle for 
cross-diffusion systems, the upper bounds can only be proven if we assume 
equal diffusion constants (A2). Under this assumption, the solvent concentration 
satisfies
$$
  \pa_t u_0 = D\diver(\na u_0 - u_0 w\na\Phi), \quad\mbox{where } 
	w = \beta\sum_{i=1}^n z_iu_i,
$$
for which a discrete maximum principle can be applied. The $L^\infty$ bounds on 
the concentrations then ensure the existence of solutions for the scheme. 
If additionally the drift term vanishes (A3), a discrete version of the entropy 
inequality, the uniqueness of discrete solutions and most importantly, 
the convergence of the scheme can be proven (under an additional regularity
assumption on the mesh). For details, we refer to \cite{CCGJ18}.


\section{Numerical experiments}\label{sec:Num}

\subsection{Implementation} 

The finite-element discretization is implemented within the finite-element 
library NGSolve/Netgen, see \cite{Sch14,Sch97}. The nonlinear equations are solved 
in every time step by Newton's method in the variables $w_i$ and $\Phi$. 
The Jacobi matrix is computed using the NGSolve function {\tt AssembleLinearization}. 
The finite-volume scheme is implemented in Matlab. Also here, the nonlinear equations 
are solved by Newton's method in every time step, using the variables $u$, $\Phi$, 
and $u_0$. 

We remark that the finite-volume scheme also performs well when we use a simpler 
semi-implicit scheme, where we compute $u$ from equation \eqref{2.equ} with 
$\Phi$ taken from the previous time step via Newton's method and subsequently only 
need to solve a linear equation to compute the update for the potential. 
It turned out that this approach is not working for the finite-element discretization.
Furthermore, the computationally cheaper implementation used in \cite{JuLe18} for
a similar scheme in one space dimension, where a Newton and Picard iteration are 
combined, did not work well in the two-dimensional test cases presented in 
this paper.


\subsection{Test case 1: calcium-selective ion channel}

Our first test case models the basic features of an L-type calcium channel
(the letter L stands for ``long-lasting'', referring to the length of activation).
This type of channel is of great biological importance, as it is present in the 
cardiac muscle and responsible for coordinating the contractions of the heart 
\cite{CSBB01}. The selectivity for calcium in this channel protein is caused by the 
so-called EEEE-locus made up of four glutamate residues. We follow the modeling 
approach of \cite{NGHE01}, where the glutamate side chains are each treated as two 
half charged oxygen ions, accounting for a total of eight $O^{1/2-}$ ions confined 
to the channel. In contrast to \cite{NGHE01}, where the oxygen ions are described by 
hard spheres that are free to move inside the channel region, we make a further 
reduction and simply consider a constant density of oxygen in the channel that 
decreases linearly to zero in the baths (see Figure \ref{fig.geom}),
$$
  u_{\rm ox}(x,y) = u_{{\rm ox},\max}\times\begin{cases}
  1 \quad &\text{for }0.45 \le x \le 0.55, \\
  10(x-0.35) \quad &\text{for }0.35 \le x \le 0.45,\\
  10(0.65-x) \quad &\text{for }0.55 \le x \le 0.65, \\
  0 \quad &\text{else},
  \end{cases}
$$
where the scaled maximal oxygen concentration equals 
$u_{{\rm ox},\max}=(N_A/u_{\rm typ})\cdot 52\,\,$mol/L. Here,
$N_A\approx 6.022\cdot 10^{23}\,$mol$^{-1}$ is the Avogadro constant and
$u_{\rm typ}=3.7037\cdot 10^{25} L^{-1}$ the typical concentration
(taken from \cite[Table 1]{BSW12}). In addition to the immobile oxygen ions, we 
consider three different species of ions, whose concentrations evolve according to 
model equations \eqref{eq:ion_trans}: calcium (Ca$^{2+}$, $u_1$), 
sodium (Na$^+$, $u_2$), and chloride (Cl$^-$, $u_3$). We assume that the oxygen ions 
not only contribute to the permanent charge density $f=-u_{\rm ox}/2$, but also take 
up space in the channel, so that we have $u_0=1-\sum_{i=1}^3 u_i-u_{\rm ox}$ for the 
solvent concentration. 

For the simulation domain, we take a simple geometric setup 
resembling the form of a channel; see Figure \ref{fig.geom}. The boundary conditions 
are as described in the introduction, with constant values for the ion concentrations 
and the electric potential in the baths. The physical parameters used in our simulations 
are taken from \cite[Table 1]{BSW12}. The simulations are performed with a constant 
(scaled) time step size $\tau=2\cdot 10^{-4}$. 
The initial concentrations are simply taken 
as linear functions connecting the boundary values. An admissible mesh consisting of 
74 triangles was created with Matlab's {\tt initmesh} command, which produces 
Delauney triangulations. Four finer meshes were obtained by regular refinement, 
dividing each triangle into four triangles of the same shape.

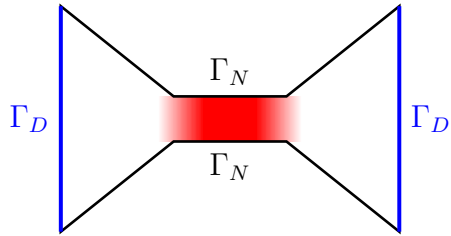
\begin{figure}[ht]
	\begin{tikzpicture}
	\shade[left color=white,right color=red] (1.3,1.2) rectangle (1.9,1.8);
	\shade[left color=red,right color=white] (2.6,1.2) rectangle (3.2,1.8);
	\draw [fill=red,red] (1.9,1.2) rectangle (2.6,1.8);
	\draw[line width=1pt] (0,0) -- (0,3) node[midway, left, blue]{$\Gamma_D$} -- (1.5,1.8) -- (3,1.8) node[midway, above]{$\Gamma_N$} -- (4.5,3) -- (4.5,0) node[midway, right, blue]{$\Gamma_D$} -- (3,1.2) -- (1.5,1.2) node[midway, below]{$\Gamma_N$} -- (0,0)--cycle;
	\draw[blue, line width=1.5pt] (0,0) -- (0,3);
	\draw[blue, line width=1.5pt] (4.5,3) -- (4.5,0);
	\end{tikzpicture}
	\caption{Schematic picture of the ion channel $\Omega$ used for the simulations. 
		Dirichlet boundary conditions are prescribed on $\Gamma_D$ (blue), 
		homogeneous Neumann boundary conditions are given on $\Gamma_N$ (black). 
		The red color represents the density of confined $O^{1/2-}$ ions.}
	\label{fig.geom}
\end{figure}

We remark that the same test case was already used in \cite{CCGJ18} to illustrate 
the efficiency of the finite-volume approximation. Furthermore, numerical simulations
for a one-dimensional approximation of the calcium channel can be found in \cite{BSW12} 
for stationary solutions and in \cite{GeJu18} for transient solutions.

Figures \ref{fig.sol} and \ref{fig.equi} present the solution to the ion transport model 
in the original variables $u$ and $\Phi$ at two different times; the first one after 
only 600 time steps and the second one after 6000 time steps, which is already close 
to the equilibrium state. The results are computed on the finest mesh with
18,944 elements. In the upper panel, the concentration profiles and electric
potential as computed with the finite-element scheme are depicted. In the 
lower panel, the difference between the finite-volume and finite-element solutions 
is plotted. We have omitted the plots for the third ion species (Cl$^-$), since 
it vanishes almost immediately from the channel due to its negative charge. 
While absolute differences are relatively small, we can still observe that the 
electric potential in the finite-element case is always smaller compared to the
finite-volume solution, while the 
peaks of the concentration profiles are more distinctive for the finite-element 
than for the finite-volume solution.

\begin{figure}[htb]
\hspace*{-10mm}\includegraphics[width=180mm]{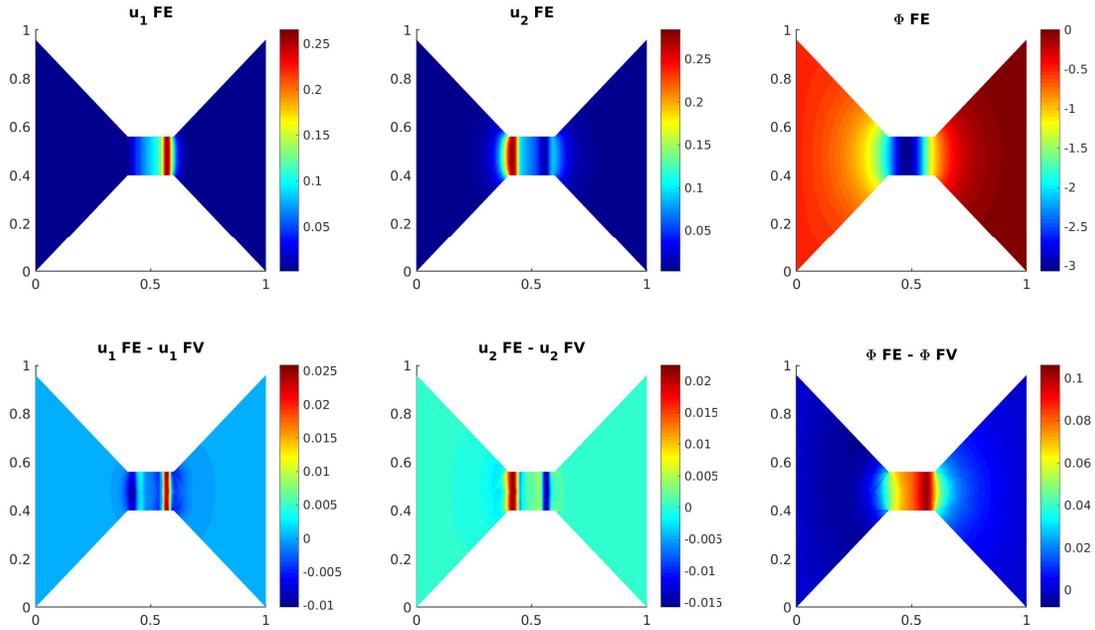}
\caption{Solution after 600 time steps computed from the finite-element scheme (top) 
and difference between the finite-volume (FV) and finite-element (FE) 
solutions (bottom).}
\label{fig.sol}
\end{figure}

\begin{figure}[htb]
\hspace*{-10mm}\includegraphics[width=180mm]{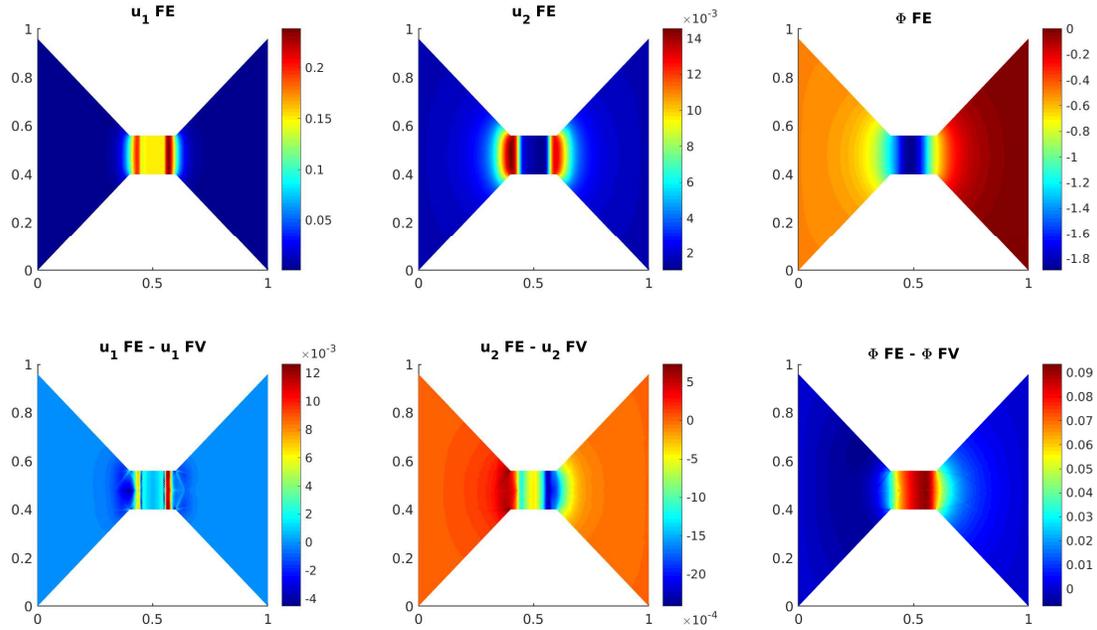}
\caption{Solution after 6000 time steps (close to equilibrium) computed from the 
finite-element scheme (top) and difference between the finite-volume (FV) and 
finite-element (FE) solutions (bottom).}
\label{fig.equi}
\end{figure}

In order to compare the two numerical methods, we test the convergence of the schemes 
with respect to the mesh diameter. Since an exact solution to our problem is not 
available, we compute a reference solution both with the finite-volume and the 
finite-element scheme on a very fine mesh with 18,944 elements and maximal cell diameter 
$h\approx 0.01$. The differences between these reference solutions in the discrete 
$L^1$ and $L^\infty$ norms are given in Table \ref{tab.ref} for the various unknowns. 
Since the finite-element and finite-volume solutions are found in different function 
spaces, one has to be careful how to compare them. The values in Table \ref{tab.ref} 
are obtained by projecting the finite-element solution onto the finite-volume space 
of functions that are constant on each cell in NGSolve, thereby introducing an 
additional error. However, the difference between the reference solutions is still 
reasonably small, especially when the simulations are already close to the 
equilibrium state. 

To avoid the interpolation error in the convergence plots, we compare the
approximate finite-element or finite-volume solutions on coarser nested meshes with 
the reference solutions computed with the corresponding method. In Figure 
\ref{fig.conv}, the errors in the discrete $L^1$ norm between the reference solution 
and the solutions on the coarser meshes at the two fixed time steps $k=600$ and $k=6000$ 
are plotted. For the finite-volume approximation, we clearly observe the expected 
first-order convergence in space, whereas for the finite-element method, the error 
decreases, again as expected, with $h^2$. These results serve as a validation for the 
theoretical convergence result proven for the finite-element scheme and show the 
efficiency of the finite-volume method even in the general case of ion transport, 
which is not covered by the convergence theorem in \cite{CCGJ18}.

\begin{table}
\begin{tabular}{l c c c c c}\hline\noalign{\smallskip}
 & $u_1$ & $u_2$ & $u_3$ & $u_0$ & $\Phi$ \\ \hline\noalign{\smallskip}
 $L^\infty$ norm, $k=600$ & 2.2405e-02 & 2.0052e-02 & 1.0319e-04 & 1.6695e-02 
 & 1.0600e-01 \\ 
 $L^1$ norm, $k=600$ & 2.2642e-04 & 3.0275e-04 & 1.3776e-05 & 2.5983e-04 & 5.1029e-03 \\ 
 $L^\infty$ norm, $k=6000$ & 1.0036e-02 & 2.3619e-03 & 1.3677e-04 & 9.1095e-03 
 & 9.5080e-02 \\  
 $L^1$ norm, $k=6000$ & 1.4161e-04 & 7.0981e-05 & 1.5498e-05 & 1.5615e-04 & 4.6543e-03
 \\ \hline\noalign{\smallskip}
\end{tabular}
\vspace{0.1cm}
\caption{Difference between the finite-volume and finite-element reference solutions 
after 600 and after 6000 time steps.}
\label{tab.ref}
\end{table}

\begin{figure}[htb]
\includegraphics[width=0.9\textwidth]{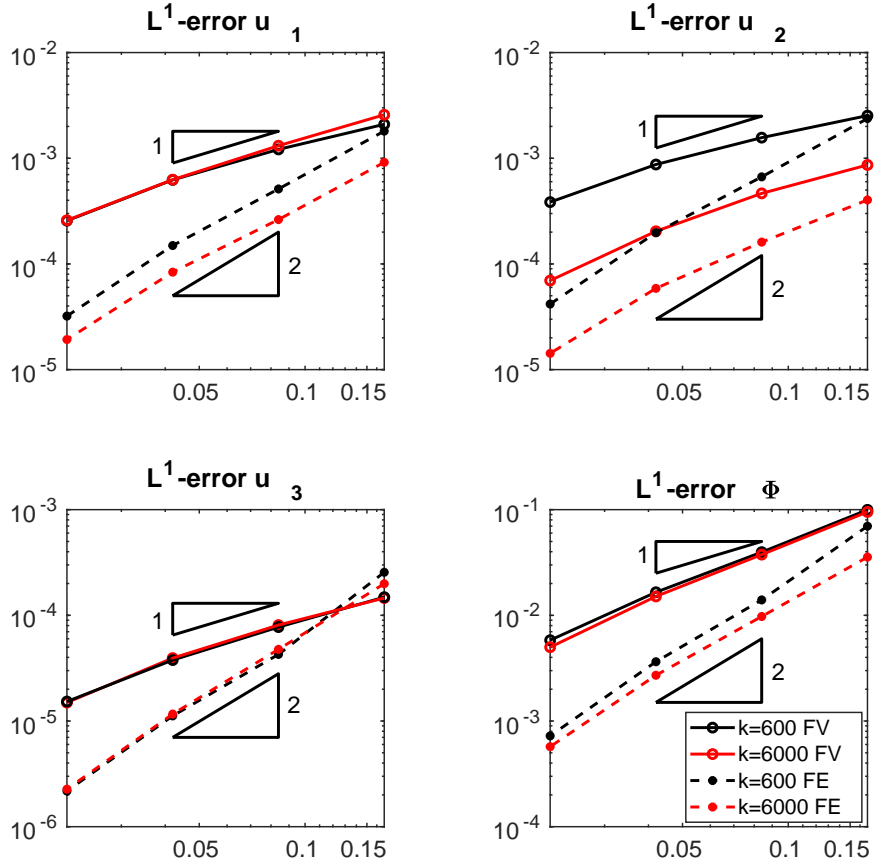}
\caption{$L^1$ error relative to the reference solution after 600 time steps (black) and 
6000 time steps (red) plotted over the mesh size $h$. Dashed lines are used for the 
finite-element solution, full lines for the finite-volume solution.}
\label{fig.conv}
\end{figure}

In Table \ref{tab.time}, the average time needed to compute one time step with the 
finite-element or finite-volume scheme for the five nested meshes is given. Clearly, 
the finite-volume scheme is much faster than the finite-element method. This is 
mostly due to the computationally expensive assembly of the finite-element matrices. 

\begin{table}
\begin{tabular}{l c c c c c}\hline\noalign{\smallskip}
 & $\T_1$ & $\T_2$ & $\T_3$ & $\T_4$ & $\T_5$ \\ \hline\noalign{\smallskip}
 FE & 2.4065e-01 & 7.9982e-01 & 2.1125e+00 & 4.9844e+00 & 17.7788e+00 \\ 
 FV & 6.7707e-03 & 2.2042e-02 & 3.0532e-01 & 1.7660e+00 & 2.2418e+00 \\  
\hline\noalign{\smallskip}
\end{tabular}
\vspace{0.1cm}
\caption{Average time needed to compute one time step (in seconds).
FE = finite-element scheme, FV = finite-volume scheme.}
\label{tab.time}
\end{table}

In addition to the convergence analysis, we also study the behavior of the discrete 
entropy for both schemes. We consider in both cases the entropy relative to the steady 
state $(u^\infty_{i},\Phi^\infty)$, which is computed from the corresponding 
discretizations of the stationary equations with the same parameters and boundary data. 
Figure \ref{fig.entropy} shows the relative entropy (see \cite[Section 6]{CCGJ18})
and the $L^1$ error compared to the equilibrium state for the finite-element 
and finite-volume solutions on different meshes. 
Whereas for the coarsest mesh the convergence rates differ notably, 
we can observe a similar behavior when the mesh is reasonably fine. 
In Figure \ref{fig.entropy_err}, 
we investigate the convergence of the relative entropy with respect to the mesh size. 
As before, we observe second-order convergence for the finite-element scheme and a
first-order rate for the finite-volume method. 

\begin{figure}[htb]
\includegraphics[width=\textwidth]{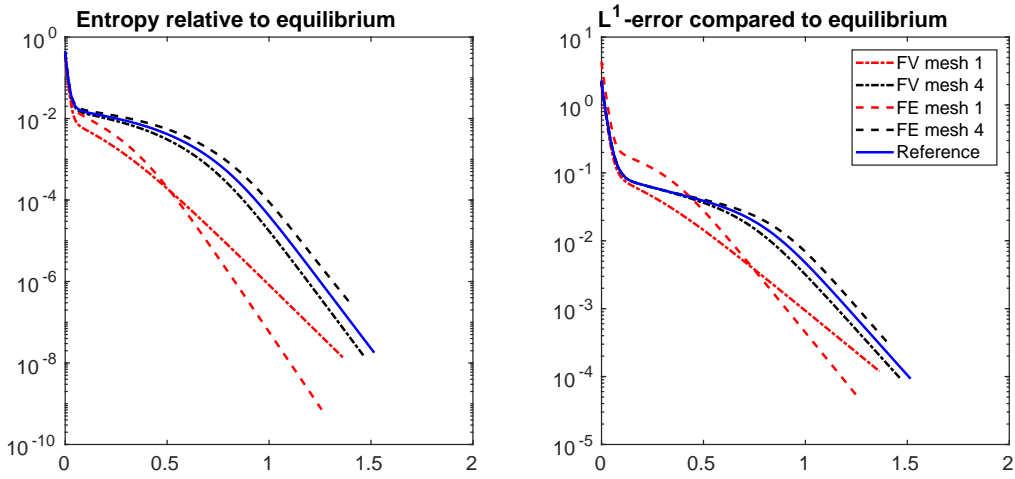}
\vspace{-1cm}
\caption{Relative entropy (left) and sum of $L^1$ differences of $u$ and $\Phi$ relative 
to the equilibrium state (right) over time for various meshes.
Mesh 1 has 74 triangles, mesh 4 has 18,944 elements.}
\label{fig.entropy}
\end{figure}

\begin{figure}[htb]
\includegraphics[width=0.9\textwidth]{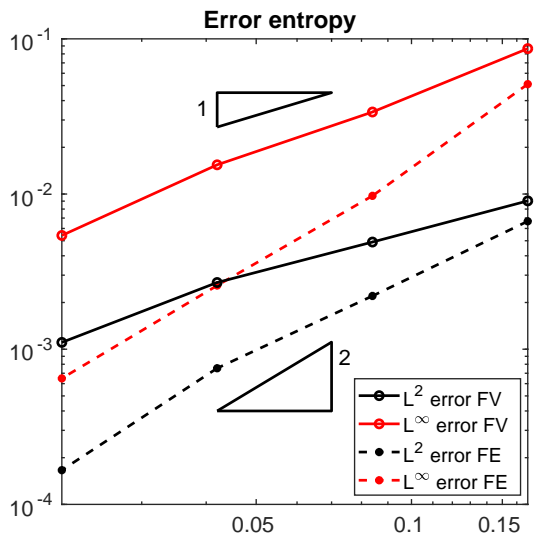}
\caption{Error for the relative entropy with respect to mesh size.}
\label{fig.entropy_err}
\end{figure}


\subsection{Test case 2: bipolar ion channel}

The second example models a pore with asymmetric charge distribution, which occurs 
naturally in biological ion channels but also in synthetic nanopores. Asymmetric pores 
typically rectify the ion current, meaning that the current measured 
for applied voltages 
of positive sign is higher than the current for the same absolute value of voltage 
with negative sign. The setup is similar to that of an N-P semiconductor diode. 
The N-region is characterized by the fixed positive charge. The anions are the 
counter-ions and thus the majority charge carriers, while the cations are the co-ions 
and minority charge carriers. In the P-region, the situation is exactly the other way 
around. In the on-state, the current is conducted by the majority carriers, while
in the off-state, the minority carriers are responsible for the current, which 
leads to the rectification behavior. 

Often, bipolar ion channels are modeled with asymmetric surface charge distributions 
on the channel walls. However, to fit these channels into the framework of our model, 
we follow the approach described in \cite{HBGVRK15}. Similar to the first test case, 
we assume that there are eight confined molecules inside the channel, but this time 
four molecules are positively charged ($+0.5e$) and the other four molecules 
are negatively charged ($-0.5e$). The simulation domain $\Omega\subset\R^2$ is depicted 
in Figure \ref{fig.omega_bip}. The shape of the domain and the parameters used for 
the simulations are taken from \cite{HBGVRK15} and are summarized in Table 
\ref{tab.parameters2}. The mesh (made up of 2080 triangles) was created with 
NGSolve/Netgen. We consider two mobile species of ions, one cation (Na$^+$, $u_1$) 
and one anion (Cl$^-$, $u_2$). The confined ions are modeled as eight fixed circles 
of radius $1.4$, where the concentration $c\equiv c_{\max}$ is such that the portion 
of the channel occupied by these ions is the same as in the simulations in 
\cite{HBGVRK15}. The solvent concentration then becomes $u_0=1-u_1-u_2-c$.

\begin{figure}[htb]
\includegraphics[width=0.5\textwidth]{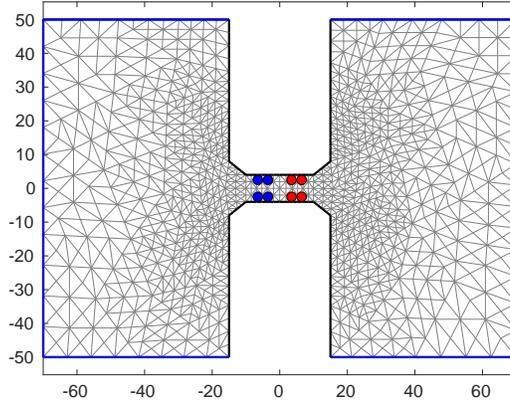}
\vspace{-1cm}
\caption{Simulation domain with triangulation for the bipolar ion channel. 
The blue circles represent positively charged confined ions, the red circles 
negatively charged ions. The black (blue) part of the boundary is equipped with 
Neumann (Dirichlet) boundary conditions.}
\label{fig.omega_bip}
\end{figure}

\begin{table}
\begin{tabular}{l l l }\hline\noalign{\smallskip}
Meaning & Value & Unit  \\ \hline\noalign{\smallskip}
Diffusion coefficients $D_1$, $D_2$ & 1 &   \\ 
Effective permittivity $\lambda^2$ & 1.1713 &  \\
Effective mobility $\beta$ & 3.8922 & \\
Bath concentrations $\overline{u_1}$, $\overline{u_2}$ & 0.0016 & \\
Confined ion concentration $c_{\max}$ \quad & 0.2971 & \\
Typical length $L_\text{typ}$ & 1e-10 & m \\
Typical concentration $u_\text{typ}$ & 3.7037e+28 & Nm$^{-3}$ \\
Typical voltage $\Phi_\text{typ}$ & 0.1 & V \\
Typical diffusion $D_\text{typ}$ & 1.3340e-9 & m$^2$s$^{-1}$ \\
  \hline\noalign{\smallskip}
\end{tabular}
\vspace{0.1cm}
\caption{Dimensionless parameters used for the simulation of the bipolar ion channel 
and values used for the scaling.}
\label{tab.parameters2}
\end{table}

By changing the boundary value $\overline{\Phi}_\text{right}$ for the potential 
$\Phi$ on the right part of the Dirichlet boundary (on the left side, it is fixed 
to zero), we can apply an electric field in forward bias (on-state, 
$\overline{\Phi}_\text{right}=1$) or reverse bias (off-state, 
$\overline{\Phi}_\text{right}=-1$). Figures \ref{fig.equilibrium_on} and 
\ref{fig.equilibrium_off} show the stationary state computed with the finite-element
method in the on- and off-state, respectively. 
Evidently, the ion concentrations in 
the on-state are much higher than in the off-state. In comparison with the results 
from \cite{HBGVRK15}, where the Poisson--Nernst--Planck equations with linear
diffusion (referred to as the linear PNP model) were combined with Local 
Equilibrium Monte--Carlo simulations, we find that with the 
Poisson--Nernst--Planck equations with cross-diffusion (referred to as the
nonlinear PNP model), the charged ions in the channel attract an amount of ions 
higher than the bath concentrations even in the off-state.

\begin{figure}[htb]
\includegraphics[width=\textwidth]{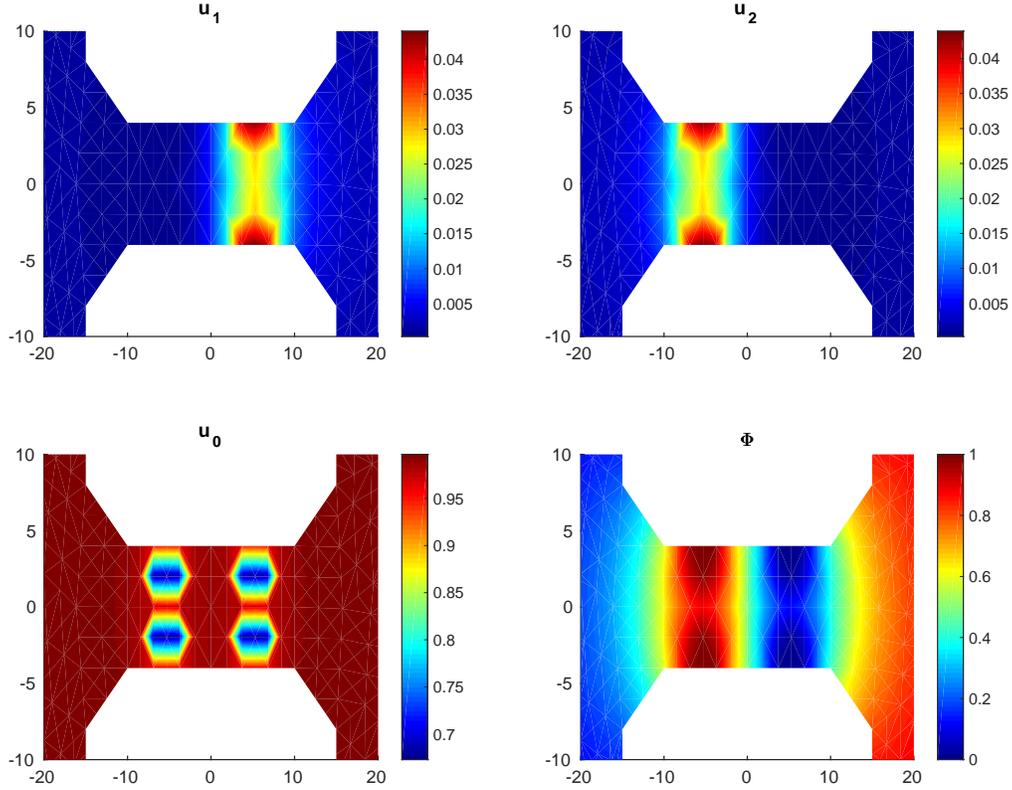}
\vspace{-1cm}
\caption{Stationary solution in the on-state (channel region).}
\label{fig.equilibrium_on}
\end{figure}

\begin{figure}[htb]
\includegraphics[width=\textwidth]{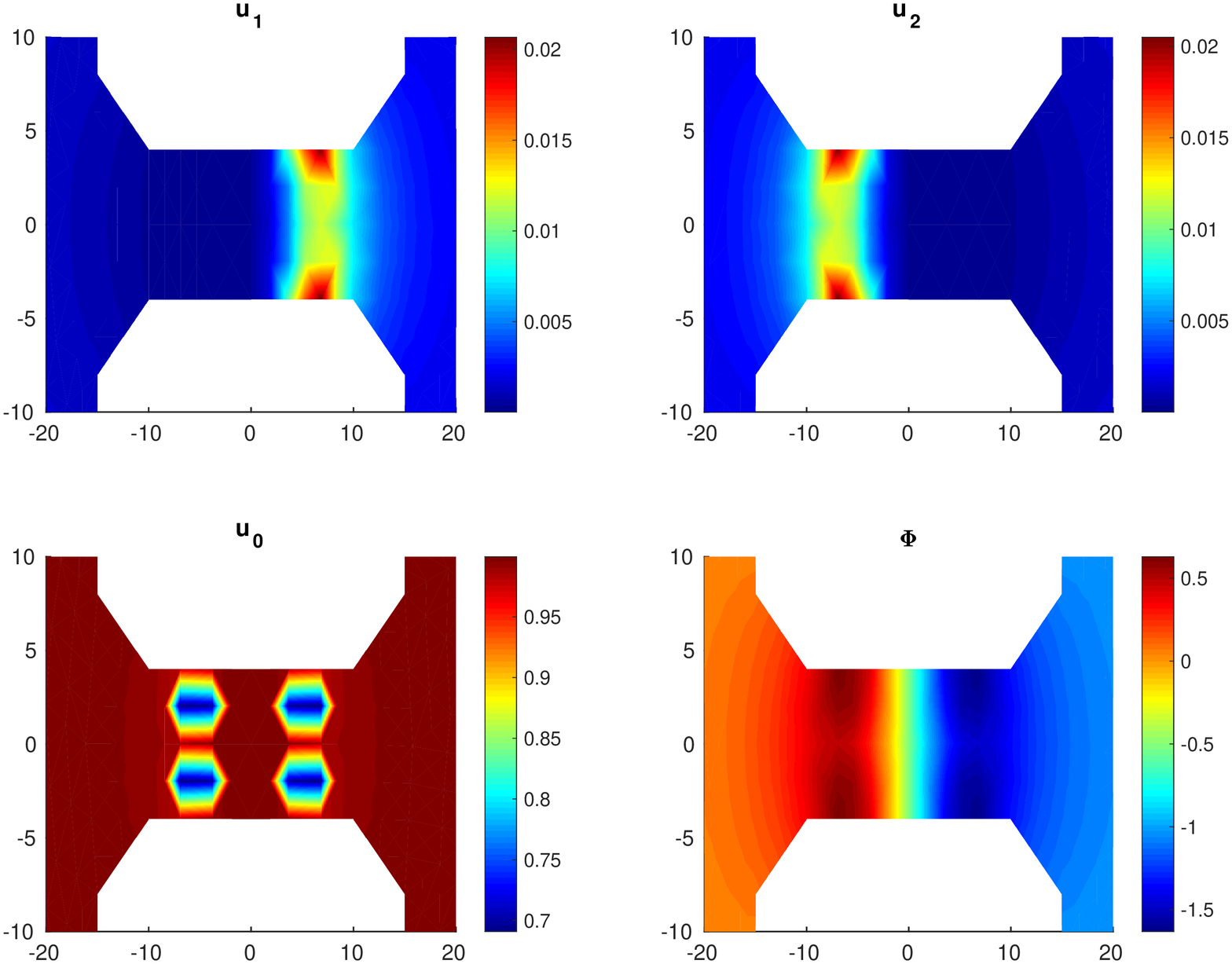}
\vspace{-1cm}
\caption{Stationary solution in the off-state (channel region).}
\label{fig.equilibrium_off}
\end{figure}

From a modeling point of view, it is an important question whether the nonlinear 
PNP model reproduces the rectification mechanism described above. 
For this purpose, we need to calculate the electric current $I$ flowing through the 
pore, given by
\begin{equation}\label{eq:current}
  I = -\sum_{i}z_i\int_{A}\F_i\cdot\nu ds,
\end{equation}
where $A$ is the cross-section of the pore and $\nu$ the unit normal to $A$. 
In the finite-element setting, we can use the representation of the fluxes in 
entropy variables, $\F_i = D_i u_i(w,\Phi)u_0(w,\Phi)\nabla w_i$ and compute 
the integrals in \eqref{eq:current} using a quadrature formula along the line $x=10$. 

\begin{figure}[htb]
\includegraphics[width=\textwidth]{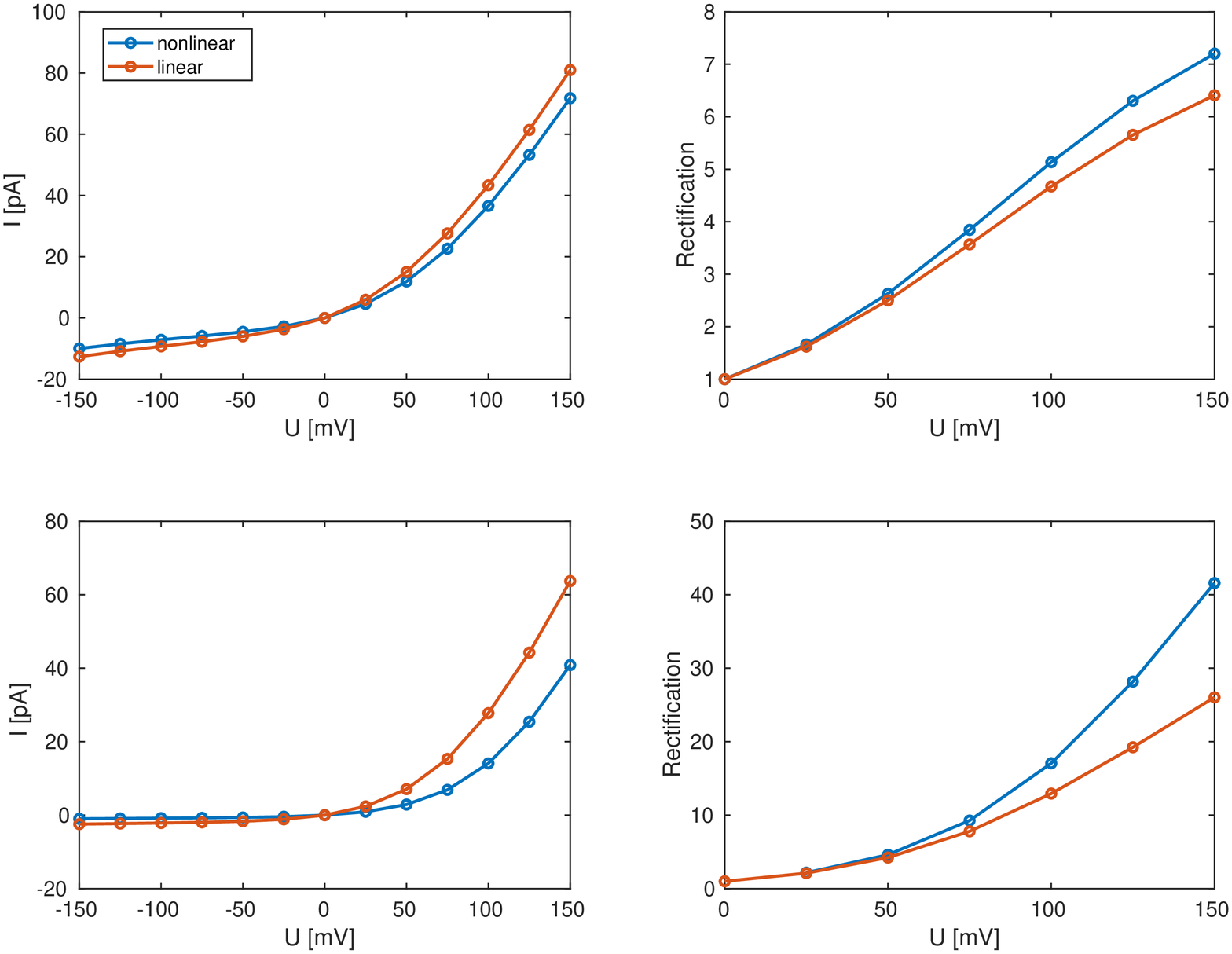}
\vspace{-1cm}
\caption{Current-voltage curves and rectification. First row: The parameters are 
as in Table \ref{tab.parameters2}; second row: with $c_{\max}=0.7$.}
\label{fig.rectification}
\end{figure}

Figure \ref{fig.rectification} shows the current-voltage curves obtained with the 
finite-element solutions. In addition, the rectification is 
depicted, which is calculated for voltages $U\ge 0$ according to
\begin{equation*}
  r(U) = \bigg| \frac{I(U)}{I(-U)} \bigg|.
\end{equation*}
We also compute the current-voltage curve for the linear PNP model, 
which is obtained from the model equations by setting $u_0\equiv 1$, 
such that
\begin{equation*}
  \pa_t u_i = \diver \big( D_i\na u_i + D_i\beta z_iu_i\na \Phi \big).
\end{equation*}
We expect from the simulations done in \cite{BSW12} for the calcium channel that 
the current of the nonlinear PNP model is lower than that one from the linear 
PNP model. This expectation is confirmed also in this case.  
As Figure \ref{fig.rectification} shows, the rectification is stronger in 
the nonlinear PNP model. The difference between the two models is even more pronounced 
when we increase the concentration of the confined ions to $c_{\max}=0.7$. In 
that case, the channel gets more crowded and size exclusion has a bigger effect. 
We observe a significantly lower current and higher rectification for the 
nonlinear PNP model. 


\section{Conclusions}\label{sec.Con}

In this work, we have presented a finite-element discretization of a cross-diffusion
Poisson--Nernst--Planck system and recalled a finite-volume scheme that was 
previously proposed and analyzed \cite{CCGJ18}. In the following, we summarize the 
differences between both approaches from a theoretical viewpoint and our findings 
from the numerical experiments.

\begin{itemize}
\item \textbf{Structure of the scheme:} The finite-element scheme strongly relies 
on the entropy structure of the system and is formulated in the entropy variables.
From a thermodynamic viewpoint, the entropy variables are related to the chemical
potentials, which gives a clear connection to nonequilibrium thermodynamics.
On the other hand, the finite-volume scheme exploits the drift-diffusion structure 
that the system displays in the original variables.
\item \textbf{$L^\infty$ bounds:} Due to the formulation in entropy variables, 
the $L^\infty$ bounds for the finite-element solutions follow immediately from 
\eqref{1.uinv} without the use of a maximum principle. In other words, the
lower and upper bounds are inherent in the entropy formulation.
In the case of the finite-volume scheme, we can apply a discrete maximum principle,
but only under the (restrictive) assumption that the diffusion coefficients $D_i$ are
the same. 
\item \textbf{Convergence analysis:} The entropy structure used in the finite-element 
scheme allows us to use the same mathematical techniques for the convergence proof
as for the continuous system, but a regularizing term has to be added to ensure the 
existence of discrete solutions. 
The convergence of the finite-volume solution requires more 
restrictive assumptions: In addition to the equal diffusion constants necessary for 
proving the existence and $L^\infty$ bounds, we can only obtain the entropy inequality 
and gradient estimates for vanishing potentials.
\item \textbf{Initial data:} Since the initial concentrations have to be transformed 
to entropy variables via \eqref{1.w_i}, the finite-element scheme can only be applied 
for initial data strictly greater than zero. The finite-volume scheme, on the other 
hand, can handle exactly vanishing initial concentrations.
\item \textbf{Experimental convergence rate:} In the numerical experiments, both 
schemes exhibit the expected order of convergence with respect to mesh size 
(even if we cannot prove any rates analytically): first-order convergence for the 
finite-volume scheme and second-order convergence for the finite-element scheme.
\item \textbf{Performance:} The numerical experiments done for this work suggest that 
the finite-element algorithm needs smaller time steps for the Newton iterations to 
converge than for the finite-volume scheme, especially when the solvent concentration
is close to zero. Furthermore, the assembly of the finite-element matrices is 
computationally quite expensive resulting in longer running times compared to the
finite-volume scheme.
\item \textbf{Mesh requirements:} A finite-volume mesh needs to satisfy the 
admissibility condition. This might be a disadvantage for simulations in 
three space dimensions.
\end{itemize}


\end{document}